\documentclass[10pt]{article}
\usepackage[utf8]{inputenc}
\usepackage{amsfonts,amsthm,hyperref,enumitem,amsmath}
\usepackage{graphicx, caption}
\usepackage{amssymb}
\usepackage{amsthm}
\usepackage{amsmath}
\usepackage{enumitem}
\usepackage{framed}
\usepackage{tikz}

\usepackage{soul}
\usepackage{changepage}
\usepackage{comment}
\usepackage{float}
\usepackage{hyperref}

\usepackage{mathtools}

\newtheorem{lemma}{Lemma}[section]
\newtheorem{theorem}[lemma]{Theorem}
\newtheorem{definition}[lemma]{Definition}
\newtheorem{question}[lemma]{Question}
\newtheorem{proposition}[lemma]{Proposition}

\newtheorem{claim}[lemma]{Claim}

\newcommand{\AM}[1]{\textcolor{purple}{ $\blacktriangleright$\ {\sf AM: #1}\
  $\blacktriangleleft$ }}
  
  \newcommand{\MPS}[1]{\textcolor{orange}{ $\blacktriangleright$\ {\sf MPS: #1}\
  $\blacktriangleleft$ }}
  \newcommand{\JH}[1]{\textcolor{blue}{ $\blacktriangleright$\ {\sf JH: #1}\
  $\blacktriangleleft$ }}

\title{Spanning trees in pseudorandom graphs via sorting networks}
\author{Joseph Hyde\thanks{Department of Mathematics and Statistics, University of Victoria, 3800 Finnerty Road, Victoria, BC V8P 5C2, Canada. Email: \texttt{josephhyde@uvic.ca}}  \and Natasha Morrison\thanks{Research supported by NSERC Discovery Grant RGPIN-2021-02511 and NSERC Early Career Supplement DGECR-2021-00047. Department of Mathematics and Statistics, University of Victoria, 3800 Finnerty Road, Victoria, BC V8P 5C2, Canada. Email: \texttt{nmorrison@uvic.ca}.}  \and Alp M\"uyesser\thanks{Department of Mathematics, University College London, WC1E 6BT, UK. Email: \texttt{alp.muyesser.21@ucl.ac.uk}} \and Mat\'ias Pavez-Sign\'e\thanks{Supported by ANID Basal Grant CMM FB210005 and by the European
Research Council (ERC) under the European Union Horizon 2020 research and innovation programme (grant agreement No. 947978) while the author was affiliated to the University of Warwick. Centro de Modelamiento Matem\'atico (CNRS IRL2807), Universidad de Chile, Santiago, Chile. Email: \texttt{mpavez@dim.uchile.cl}}}
\date{}
\begin{document}

\newcommand{\eps}{\varepsilon}
\newcommand{\N}{\mathbb{N}}

\newcommand{\out}{\text{\upshape out}}
\newcommand{\inn}{\text{\upshape in}}

\maketitle
\begin{abstract}
    We show that $(n,d,\lambda)$-graphs with $\lambda=O(d/\log^3 n)$ are universal with respect to all bounded degree spanning trees. 
    This significantly improves upon the previous best bound due to Han and Yang of the form $\lambda=d/\exp{(O(\sqrt{\log n}))}$, 
    and makes progress towards a problem of Alon, Krivelevich, and Sudakov from 2007. 
    \par Our proof relies on the existence of sorting networks of logarithmic depth, 
    as given by a celebrated construction of Ajtai, Koml\'os and Szemer\'edi.
    Using this construction, 
    we show that the classical vertex-disjoint paths problem can be solved for a set of vertices fixed in advance. 
\end{abstract}

\section{Introduction}

How pseudorandom does a graph need to be before it contains a certain spanning subgraph? This problem is not very well understood, especially in comparison to its purely random analogue concerning  ``thresholds'' in $\mathbb{G}(n,p)$ (see, e.g.~\cite{frankston2021thresholds}). For example, the best possible condition that forces an $(n,d,\lambda)$-graph to contain a Hamilton cycle is not known, where an $(n,d,\lambda)$-graph is an $n$-vertex $d$-regular graph such that the second largest eigenvalue (in absolute value) of the adjacency matrix is bounded above by $\lambda$. 
The well-known expander mixing lemma states that the smaller $\lambda$ is, the more pseudorandom an $(n,d,\lambda)$-graph becomes, 
in the sense that the edges are more evenly distributed across the graph (we refer the reader to \cite{krivelevich2006pseudo} for a thorough exposition). 
Intuitively, the more pseudorandom a graph is, 
the easier it becomes to find a target subgraph. However, finding optimal conditions that force $(n,d,\lambda)$-graphs to contain certain spanning subgraphs is notoriously difficult as, for instance, the only cases that are fully understood are perfect matchings \cite{krivelevich2006pseudo} and triangle-factors \cite{morrisfactors}.
\par Our main motivation in this paper is the following problem of Alon, Krivelevich, and Sudakov~\cite{alon2007embedding} from 2007. 
\begin{question}\label{mainquestion}
    Is it true that for any $\Delta\in \mathbb{N}$, there exists a constant $C=C(\Delta)$ such that any $(n,d,\lambda)$-graph with $\lambda \leq d/C$ contains \textit{every} tree on $n$ vertices with maximum degree at most $\Delta$?  
\end{question}
\par At the moment, a positive answer to this question seems to be out of reach. Indeed, the best known bound on $\lambda$ that guarantees a Hamilton path is $\lambda \leq d/O(\log^{1/3} n)$ due to a recent result of Glock, Munh{\'a} Correia,
and Sudakov \cite{glock2023hamilton} (see also \cite{KrivelevichHamilton}). In the regime of Question~\ref{mainquestion} (when $\lambda \simeq o(d)$), we only have results for restricted classes of trees. Indeed, a recent result of the fourth author~\cite{pavezsigne2023spanning} shows that if $G$ is an $(n,d,\lambda)$-graph with $\lambda\le d/C$ for some large constant $C$, then $G$ contains all $n$-vertex bounded degree trees with linearly many leaves. For general bounded degree trees, Han and Yang~\cite{han2022spanning} made some progress towards a full answer to Question~\ref{mainquestion} by showing that the condition $\lambda \leq d/ \Delta^{O(\sqrt{\log n})}$ is sufficient. 
From now on, we say that a graph is $\mathcal T(n,\Delta)$-\textit{universal} if it contains every $n$-vertex tree with maximum degree at most $\Delta$. 
\par In this paper, we make another step towards a positive answer to Question~\ref{mainquestion}. 
\begin{theorem}\label{thm:mainthm} For all $\Delta\in\mathbb N$, there exists a positive constant $C$ such that the following holds. Every $(n,d,\lambda)$-graph with $\lambda\leq d/C \log^{3} n$ is $\mathcal T(n,\Delta)$-universal.
\end{theorem}
Note that, up to polylogarithmic terms, our bound matches the best known bound on $\lambda$ for guaranteeing the existence of Hamilton paths. 
Furthermore, as with the result of Han and Yang, this result confirms the existence of $\mathcal T(n,\Delta)$-universal graphs with arbitrarily large girth,
thereby addressing a problem of Johannsen, Krivelevich, and Samotij \cite{johannsen2013expanders} (we refer the reader to \cite{han2022spanning} for further details). 

\par When proving a statement such as Theorem~\ref{thm:mainthm}, the usual strategy is to combine methods for embedding an almost-spanning tree with methods for turning an almost spanning tree into a spanning tree. For example, if the tree has many leaves, one can embed all of the tree except the leaves in some convenient way, and then use Hall's theorem in order to find a matching between the image of the parents of leaves and the leftover vertices in the host graph to complete the embedding. When the tree has few leaves, one typically uses instead its (necessarily existing) path-like substructures to finish the embedding. To achieve this, one needs to find a way to connect a given collection of $\Theta(\frac{\lambda n}{d})$ pairs of vertices with vertex-disjoint paths of length $O(\frac{d}{\lambda})$, while partitioning some target set in the process. In his breakthrough work establishing the threshold for the random graph $\mathbb{G}(n,p)$ to be $\mathcal T(n,\Delta)$-universal, Montgomery~\cite{montgomery2019spanning} developed what became known as the \textit{distributive absorption method} in order to solve a suitable version of this vertex-disjoint paths problem. 
Since then, this method has found numerous applications for embedding spanning subgraphs (see, e.g. \cite{montgomery2021proof, morrisfactors}) and, notably, Han and Yang \cite{han2022spanning} rely on this method to prove their aforementioned result.

\par Our proof, on the other hand, does not rely on any absorption technique. Instead, we take a new approach to address this problem in sparse pseudorandom graphs, by carefully constructing a special graph which is based on the existence of optimal sorting networks (as given by the celebrated result by Ajtai, Koml\'os, and Szemer\'edi~\cite{ajtai19830}) in combination with the rollback method to embed trees into sparse expanders (which is based on an idea of Johannsen \cite{draganic2022rolling}). This allows us to show that the vertex-disjoint paths problem (see \cite{draganic2022rolling}) in sparse expanders can be solved with a set of vertices \textit{fixed in advance} (see Lemma~\ref{lemma:find_sorting_network_in_expander}). We believe that this construction might be of independent interest, and we include another application in Section~\ref{sec:concluding} showing that sparse pseudorandom graphs can be factorised\footnote{Given graphs $F$ and $G$, an $F$-\textit{factor} in $G$ is a collection of vertex-disjoint copies of $F$ covering every vertex in $G$.} into cycles of polylogarithmic length. 

\section{Proof overview}\label{proof_overview}

Our starting point is the above mentioned result due to the fourth author, which states that if $\lambda\leq d/C$, for some large constant $C$, then an $(n,d,\lambda)$-graph contains every bounded degree tree with $\Omega(\frac{\lambda n}{d})$ leaves.
\begin{theorem}[{\cite[Theorem~4.2]{pavezsigne2023spanning}}]\label{thm:manyleaves}
For all $\Delta\in\mathbb N$, there exist positive constants $C$ and $K$ such that the following holds for all sufficiently large $d\in\mathbb N$.
If $G$ is an $(n,d,\lambda)$-graph with $\lambda\leq d/C$,
then $G$ contains a copy of every $n$-vertex tree with maximum degree at most $\Delta$ and at least $\frac{K\lambda n}{d}$ leaves. 
\end{theorem}
For a graph $H$, say that a path $P\subset H$ is a \textit{bare path} if all vertices in $P$ have degree $2$ in $H$. The following well-known result of Krivelevich \cite{Krivelevichtrees} implies that trees with few leaves contain many bare paths.
\begin{lemma}\label{lemma:fewleaves}Let $n,k,\ell\in\mathbb N$ and let $T$ be an $n$-vertex tree with at most $\ell$ leaves.
Then $T$ contains a collection of at least $\frac{n}{k+1}-(2\ell-2)$ vertex-disjoint bare paths, each of length $k$. 
\end{lemma}
Then, combining Theorem~\ref{thm:manyleaves} and Lemma~\ref{lemma:fewleaves},
we see that when proving Theorem~\ref{thm:mainthm} we need only concern ourselves with trees $T$ containing $\Omega(n/\log^3 n)$ bare paths, each of length $O(\log^3 n)$ (see the proof of Theorem~\ref{thm:mainthm} for the formal calculation). 
Once the internal vertices of the bare paths are removed from $T$, the resulting forest, say $F$, is not spanning anymore, in which case there are well-known results (see~\cite{FP1987,H2001}) that allow us to find, in an expander graph, a copy of every almost spanning forest with maximum degree at most $\Delta$. To complete the embedding of $T$, we thus only need to find a collection of vertex-disjoint paths of length $O(\log^3n)$ joining the image of the endpoints of the bare paths we just removed from $T$, and moreover, while doing so, using all the leftover vertices in $G$. What we would like to do is sequentially find  perfect matchings between $O(\log^3 n)$ consecutive pairs of sets, each of size $\Theta(n/\log^3 n)$ (setting aside some random sets in the beginning and using Lemma~\ref{lemma:matching} can achieve this, for example). 
The union of these matchings would then form a path-factor which corresponds to the desired collection of bare paths.
The problem with this argument is that the bare paths we need to embed have designated endpoints, since we have already committed to an embedding of the tree with the internal vertices of the bare paths removed. That is, how can we be sure in the very small amount of our $(n,d,\lambda)$-graph we have at the end of our embedding process that the exact path-factor corresponding to the remaining bare paths exists?
The next lemma is designed to handle this complication (we refer the reader to Section~\ref{prelims} for the definitions of $m$-joined, $(D,m)$-extendable and $I(X)$). 

\begin{lemma}\label{lemma:find_sorting_network_in_expander} There is an absolute constant $C_{\ref{lemma:find_sorting_network_in_expander}}$ with the following property. Let $1/n\ll 1/K\ll 1/C_{\ref{lemma:find_sorting_network_in_expander}}$, and let $D,m\in \mathbb{N}$ satisfy $m\leq n/100D$ and $D\ge 100$. Let $G$ be an $m$-joined graph on $n$ vertices which contains disjoint subsets $V_1, V_2\subseteq V(G)$ with $|V_1|=|V_2|\leq n/K\log^{3}n$, and set $\ell:=\lfloor C_{\ref{lemma:find_sorting_network_in_expander}} \log^3 n \rfloor$.
Suppose that $I(V_1\cup V_2)$ is $(D,m)$-extendable in $G$.
\par Then, there exists a $(D,m)$-extendable subgraph $S_{res}\subseteq G$ such that for any bijection $\phi\colon V_1\to V_2$, there exists a $P_\ell$-factor of $S_{res}$ where each copy of $P_\ell$ has as its endpoints some $v\in V_1$ and $\phi(v)\in V_2$.
\end{lemma}
\par We remark that Lemma~\ref{lemma:find_sorting_network_in_expander} can be used in $(n,d,\lambda)$ graphs with $\lambda/d=O(1)$, and thus might have applications beyond those which we consider in the current paper. We will return to this aspect in the concluding remarks.
\par To prove Theorem~\ref{thm:mainthm}, we proceed (roughly) as follows. We first use Lemma~\ref{lemma:find_sorting_network_in_expander}, where $V_1$ and $V_2$ are two random subsets of size $\Theta(n/\log^3 n)$, getting a subgraph $S_{res}$ which we will use at the end of the proof in order to complete the embedding of $T$. We then find a copy of the forest $F$ (using Lemma~\ref{lemma:extendable:embedding}) in $G\setminus V(S_{res})$. The next step is to find a collection of consecutive matchings between $O(\log^3n)$ sets of size $\Theta(n/\log^3n)$, starting from the image of the endpoints of the bare paths and ending at $V_1$ and $V_2$, respectively.  We achieve this by first picking pairwise disjoint random subsets $V'_1,\ldots, V'_t$, with $t=O(\log^3 n)$ and  $|V'_i|=\Theta(n/\log^3n)$ for all $i\in [t]$, so that every vertex in the graph has many neighbours into each of those sets. Then, after $F$ is embedded, we distribute all the leftover vertices into $\bigcup V'_i$ so that each $V_i'$ has the same size. Using that every vertex has good degree into each $V'_i$, combined with the expansion properties of $(n,d,\lambda)$-graphs, we can find a perfect matching between $V'_i$ and $V'_{i+1}$ for each $1\le i<t$ (see Lemma~\ref{lemma:matching}). The union of these matchings will then give us a collection of paths connecting, in some order, the image of the endpoints of the bare paths we have previously removed with $V_1$ and $V_2$. Finally, we use the property of $S_{res}$ to partition $V(S_{res})$ into paths of the same length, connecting the vertices of $V_1$ and $V_2$ in whatever order we need to finish the embedding of $T$.

Let us emphasise that the strength of Lemma~\ref{lemma:find_sorting_network_in_expander} comes from the fact that $S_{res}$ is fixed with respect to $V_1$ and $V_2$ only, and does not depend on the choice of the bijection $\phi:V_1\to V_2$.
It is worth comparing Lemma~\ref{lemma:find_sorting_network_in_expander} with \cite[Theorem 2]{draganic2022rolling} due to Dragani\'c, Krivelevich and Nenadov, which has a similar statement, 
but the order of $S_{res}$ and $\phi$ is reversed in the quantification, making their statement weaker. 
On the other hand, the result from \cite{draganic2022rolling} is quantitatively stronger than Lemma~\ref{lemma:find_sorting_network_in_expander}, in the sense that it works for sets $V_1$ and $V_2$ of size as large as $\Theta(n/\log n)$ (which is optimal as typically we cannot expect to find paths of length much shorter than $\log n$ in sparse expanders).
Getting similar quantitative improvements for Lemma~\ref{lemma:find_sorting_network_in_expander} would be interesting, as it would immediately translate to improvements in the bounds we are obtaining for Theorem~\ref{thm:mainthm}. The $\log^3 n$  factor in our bound comes from $S_{res}$ being constructed from a sorting network of depth $O(\log n)$ where the comparison gadgets are replaced with subgraphs of size $O(\log^2 n)$ coming from Lemma~\ref{lemma:gadgets_exist_more_abstract}. 

\par Pushing our methods further to show that Theorem~\ref{thm:mainthm} holds under the weaker hypothesis that $\lambda \leq d/\log^{1+o(1)}n$ seems challenging, but possible. Going beyond this, on the other hand, would possibly require entirely new ideas. 


\section{Preliminaries}\label{prelims}
\subsection{Notation}\label{notation}
For a graph $G$, let $V(G)$ and $E(G)$ denote the vertex set and edge set of $G$, respectively, and write $|G|:=|V(G)|$. For a vertex $v\in V(G)$, let $N(v)$ denote the neighbourhood of $v$ and, for a subset $U\subset V(G)$, let $N(v,U):=N(v)\cap U$. We let $d(v):=|N(v)|$ denote the degree of $v$ and write $d(v,U):=|N(v,U)|$ for the degree of $v$ into a subset $U\subset V(G)$. The minimum degree of $G$ is denoted by $\delta(G)$, and we write $\Delta(G)$ for the maximum degree of $G$. For $U\subset V(G)$, let $\Gamma(U):=\bigcup_{u\in U}N(u)$ denote the {\it neighbourhood of $U$} and let $N(U):=\Gamma(U)\setminus U$ be the {\it external neighbourhood of $U$}. If necessary, we will add subscripts to denote which graph we are working with. For a subset $S\subseteq V(G)$, let $G[S]$ denote the graph induced by $S$ and, given a subset $S'\subset V(G)\setminus S$, let $G[S,S']$ denote the bipartite graph with bipartition $S\cup S'$ and edges of the form $ss'\in E(G)$ with $s\in S$ and $s'\in S'$, further letting $e(S,S')$ denote the number of edges in $G[S,S']$. We write $I(S)$ for the edgeless subgraph with vertex set $S$, and $G-S:=G[V(G)\setminus S]$. 

A path $P$ in $G$ is a sequence of distinct vertices $P=v_1\ldots v_t$ such that $v_{i}v_{i+1}\in E(G)$ for each $1\le i<t$, in which case we say that $v_1$ and $v_t$ are the endpoints of $P$ and $v_2,\ldots, v_{t-1}$ are the internal vertices of $P$. The length of $P$ is equal to its number of edges. Given two distinct vertices $u,v\in V(G)$, a $(u,v)$-path is a path whose endpoints are precisely $u$ and $v$. For a subgraph $H\subset G$ and an edge $e\in E(G)$, we let $H+e$ denote the graph with edge set $E(H)\cup e$. Moreover, if $P=v_1\ldots v_t$ is a path and we denote $P_i=v_1\ldots v_i$ for $1\le i\le t$, then we define $H+P_i:=(H+P_{i-1})+v_{i-1}v_i$.

For a positive integer $n$, we let $[n]=\{1,\ldots, n\}$ denote the set of the first $n$ positive integers. Given real numbers $a,b,c$, we write $a=b\pm c$ to denote that $b-c\le a\le b+c$. We use standard notation for ``hierarchies'' of constants, writing $x\ll y$ to mean that there is a non-decreasing function $f : (0,1] \rightarrow (0, 1]$ such that all relevant subsequent statements hold for $x\leq f(y)$. Hierarchies with multiple constants are defined similarly.  We omit rounding signs where they are not crucial.

\subsection{Concentration bounds}

We need the following concentration bound, which is a simple corollary of a result of McDiarmid (stated as Lemma 6.1 in \cite{liebenau2023asymptotic}).

\begin{lemma}\label{lem:concentration} Let $n$ be sufficiently large and let $G$ be an $n$-vertex graph with $\delta(G)\geq \log^6 n$.
Let $R$ be a uniformly random subset of $V(G)$ of size $k\geq n/\log^4 n$, and let $v\in V(G)$. Then, \[\mathbb{P}\left(d(v, R)=(1\pm 1/10)d(v)\cdot\tfrac{k}{n}\right)\geq 1-1/n^2.\]
\end{lemma}


\subsection{Properties of $(n,d,\lambda)$-graphs}
The next three results are standard properties of $(n,d,\lambda)$-graphs (see \cite{krivelevich2006pseudo}).
\begin{lemma}\label{lemma:secondeig}Every $(n,d,\lambda)$-graph satisfies $\lambda\ge \sqrt{d\cdot\frac{n-d}{n-1}}$.    
\end{lemma}
\begin{lemma}[Expander Mixing Lemma]\label{lemma:mixing}Let $G$ be an $(n,d,\lambda)$-graph.
Then, for every pair of (not necessarily disjoint) sets $A,B\subset V(G)$,
we have 
\[\left|e(A,B)-\tfrac{d}{n}|A||B|\right|<\lambda\sqrt{|A||B|}.\]
\end{lemma}
We say a graph $G$ is $m$-\textit{joined} if $e(A,B)\geq 1$ for any disjoint sets $A,B\subseteq V(G)$ with $|A|,|B|\geq m$.

\begin{lemma}\label{lemma:joined}Every $(n,d,\lambda)$-graph is $\frac{\lambda n}{d}$-joined.
\end{lemma}
The following result, which is a simple corollary of the expander mixing lemma, will allow us to translate minimum degree conditions into an expansion property for small sets.
\begin{lemma}\label{lemma:expansion2}Let $\mu,C,D,z>0$ and $n \in \mathbb{N}$ such that $\mu Cz> 4D$, and let $d\in\mathbb N$ and $\lambda>0$ satisfy $d/\lambda\ge C\log^3 n$.
Suppose $G$ is an $(n,d,\lambda)$-graph which contains subsets $X,Y\subset V(G)$ such that for every $v\in X$,
$d(v,Y)\ge\mu dz/\log^3 n$. 
Then, every subset $S\subset X$ of size $|S|\le \lambda n/d$ satisfies $|N(S)\cap Y|\ge D|S|$.
\end{lemma}
\begin{proof}Suppose that there exists a subset $S\subset X$ of size $1\le |S|\le \frac{\lambda n}{d}$ such that $|N(S)\cap Y|<D|S|$. Let $Z=N(S)\cap Y$.
Lemma~\ref{lemma:mixing} implies that
\[\frac{\mu dz}{\log^3n}|S|\le e(S,Z)\le \frac{d}{n}|S||Z|+\lambda \sqrt{|S||Z|}\le \lambda D|S|+\lambda |S|\sqrt{D}.\]
Then we have 
\[\frac{\mu dz}{\log^3n}\le \lambda D+\lambda \sqrt{D}\le 4\lambda D,\]
which contradicts that $d/\lambda\ge C\log^3 n$, as $\mu Cz > 4D$.
\end{proof}


The last result that we need gives a sufficient condition to find perfect matchings between small subsets of $(n,d,\lambda)$-graphs. 
\begin{lemma}\label{lemma:matching} Let $1/C\ll\varepsilon\ll 1$ and let $n,d\in\mathbb N$,
$n\ge 3$ and $\lambda >0$ satisfy $d/\lambda\ge C\log^3n$.
Suppose $G$ is an $(n,d,\lambda)$-graph that contains disjoint subsets $A,B\subset V(G)$ with $|A|=|B|$ such that $\delta (G[A,B])\ge \varepsilon d/2\log^3n$. 
Then, $G[A,B]$ contains a perfect matching.    
\end{lemma}
\begin{proof} As $\delta(G[A,B])\ge \varepsilon d/2\log^3 n$ and $1/C\ll \varepsilon$,
using Lemma~\ref{lemma:expansion2}, with $z_{\ref{lemma:expansion2}} = 1$, $D=2$, and $\mu=\eps/2$, we may add the following expansion properties:
\begin{itemize}
    \item for every subset $S\subset A$ with $|S|\le \frac{\lambda n}{d}$, $|N(S)\cap B|\ge 2|S|$, and
    \item for every subset $S\subset B$ with $|S|\le \frac{\lambda n}{d}$, $|N(S)\cap A|\ge 2|S|$.
\end{itemize}
We will show that Hall's matching criterion holds. 
For contradiction, suppose that there exists a subset $S\subset A$ with $|N(S)\cap B|<|S|$. 
Clearly, we have $|S|>\frac{\lambda n}{d}$. 
Then, by Lemma~\ref{lemma:joined}, $G[A,B]$ is $\frac{\lambda n}{d}$-joined and thus 
\begin{equation}\label{eq:hall:1}|S|>|N(S)\cap B|\ge |B|-\tfrac{\lambda n}{d}.\end{equation}
Letting $X=A\setminus S$ and $Y=B\setminus N(S)$, we have that $|X|\le |Y|\le \frac{\lambda n}{d}$. 
Then using the expansion properties we see that $|X|\ge |N(Y)\cap A|\ge 2|Y|$, a contradiction.
Then, $G[A,B]$ satisfies Hall's matching criterion and therefore $G[A,B]$ has a perfect matching.
\end{proof}
\subsection{The extendability/rollback method}
All the results we cite below can be found in~\cite{montgomery2019spanning}. The first four are typically described as the Friedman-Pippenger tree embedding technique, subsequently developed by Haxell~\cite{H2001}.
The basic idea is that a cleverly defined inductive hypothesis (Definition~\ref{def:dmextendable}) can be maintained while extending an embedding by the addition of a leaf.
\par An idea that can be traced back to Johannsen states that this process is in fact reversible,
meaning that $(D,m)$-extendability -- defined below -- can also be maintained by the removal of leaves (this is called a `rollback' in \cite{draganic2022rolling}). Using this idea, one can efficiently find paths between prescribed vertices in an extendable subgraph, by iteratively exploring the neighbourhoods of the respective vertices until we find an overlap, and then rolling back by removing all the unused vertices in the process (see Lemma~\ref{lemma:connecting} for the precise statement).

\begin{definition}\label{def:dmextendable}Let $D,m\in\mathbb N$ with $D\ge 3$.
Let $G$ be a graph and let $S\subset G$ be a subgraph with $\Delta(S) \leq D$.
We say that $S$ is $(D,m)$-extendable 
if for all $U\subset V(G)$ with $1\le |U|\le 2m$ we have
    \begin{equation}\label{def:extendability}
        |\Gamma_G(U)\setminus V(S)|\ge (D-1)|U|-\sum_{u\in U\cap V(S)}(d_S(u)-1).
    \end{equation}
\end{definition}
The following result says that it is enough to control the external neighbourhood of small sets in order to verify extendability.

\begin{proposition}\label{prop:weakerextendability}Let $D,m\in\mathbb N$ with $D\ge 3$.
Let $G$ be a graph and let $S\subset G$ be a subgraph with $\Delta(S)\le D$. 
If for all $U\subset V(G)$ with $1\leq|U|\leq 2m$ we have 
\[|N_G(U)\setminus V(S)|\geq D|U|,\]
then $S$ is $(D,m)$-extendable in $G$.
\end{proposition}


The following result states that we can embed nearly spanning trees in an extendable way. 
\begin{lemma}
[{\cite[Corollary 3.7]{montgomery2019spanning}}]\label{lemma:extendable:embedding}
Let $D,m\in\mathbb N$ with $D\ge 3$, and let $G$ be an $m$-joined graph. 
Let $T$ be a tree with $\Delta(T)\le D/2$ and let $H$ be a $(D,m)$-extendable subgraph of $G$ with maximum degree at most $D/2$. 

If $|H|+|T|\le|G|-(2D+3)m$, 
then for every vertex $t\in V(T)$ and $v\in V(H)$,
there is a copy $S$ of $T$ in $G-V(H-v)$ in which $t$ is copied to $v$ and,
moreover, $S\cup H$ is a $(D,m)$-extendable subgraph of $G$.
\end{lemma}
The following is a key result, proved by using the rollback idea of Johannsen described earlier.
\begin{lemma}[{\cite[Corollary 3.12]{montgomery2019spanning}}]\label{lemma:connecting}Let $D,m\in\mathbb N$ with $D\ge 3$, and let $k=\lceil \log (2m)/\log (D-1)\rceil$.
Let $\ell\in\mathbb N$ satisfy $\ell\ge 2k+1$ and let $G$ be an $m$-joined graph which contains a $(D,m)$-extendable subgraph $S$ of size $|S|\le |G|-10Dm-(\ell-2k-1)$.

Suppose that $a$ and $b$ are two distinct vertices in $S$ with $d_S(a),d_S(b)\le D/2$. 
Then, there exists an $a,b$-path $P$ of length $\ell$ such that 
(i) all internal vertices of $P$ lie outside $S$, and 
(ii) $S+P$ is $(D,m)$-extendable.
\end{lemma}


\section{Sorting networks}\label{sec:sorting}
\par A {\it parallel comparison network} is a pair $(R, \mathcal{C})$ which consists of a set $R := \{r_1, r_2, \ldots, r_n\}$ of $n$ registers,
and a sequence of sets $\mathcal{C}=C_1, C_2, C_3, \ldots$,
where each $C_i$ is a collection of disjoint pairs of registers $(r_i, r_j)$ with $i < j$, i.e., for any  $(r_{i_1}, r_{i_2}), (r_{i_3}, r_{i_4}) \in C_i$ we have that $i_1, i_2, i_3, i_4$ are all distinct.
The {\it depth} of such a network is defined to be the length of the sequence $\mathcal{C}$.

\par Let $\rho_0\colon R\to [n]$ be a bijection assigning values to each register. For some $1\le i\leq \ell$, suppose, inductively, that we have defined bijections $\rho_0$, $\rho_1$, $\ldots$ , $\rho_{i-1}:R\to [n]$, and let us define $\rho_i\colon R\to [n]$ as follows.
For distinct indices $j_1, j_2 \in [n]$, we let $\rho_i(r_{j_1}) = \rho_{i-1}(r_{j_2})$ and $\rho_i(r_{j_2}) = \rho_{i-1}(r_{j_1})$ if both $(r_{j_1}, r_{j_2}) \in C_i$ and $\rho_{i-1}(r_{j_1}) > \rho_{i-1}(r_{j_2})$,
and, otherwise, $\rho_i(r_{j_1}) = \rho_{i-1}(r_{j_1})$ and $\rho_i(r_{j_2}) = \rho_{i-1}(r_{j_2})$. 
Further, if there exists any $j\in [n]$ such that $r_j \notin (r_{j_1}, r_{j_2})$ for any $(r_{j_1}, r_{j_2}) \in C_i$, then we have $\rho_i(r_{j}) = \rho_{i-1}(r_{j})$.
That is, $\rho_i$ swaps the values in registers $r_{j_1}$ and $r_{j_2}$ given by $\rho_{i-1}$ if we have both that this pair of registers was in $C_i$ and the values given to them by $\rho_{i-1}$ decrease from $r_{j_1}$ to $r_{j_2}$. Otherwise, $\rho_i$ does not change the values.
\par We say that a parallel comparison network $(R, \mathcal{C})$ with depth $\ell$ is a {\it parallel sorting network with depth $\ell$} if for \emph{any} initial assignment $\rho_0\colon R\to [n]$,
we have that $\rho_{\ell} = I$ where $I:R \to [n]$ is defined as $I(r_i) = i$ for all $i\in[n]$.

We need the following result of Ajtai, Koml\'{o}s and Szemer\'{e}di~\cite{ajtai19830} that gives a parallel sorting network with depth logarithmic in the number of registers.

\begin{theorem}[Ajtai-Koml\'os-Szemer\'edi~\cite{ajtai19830}]\label{thm:AKS_sorting} There exists an absolute constant $C_{\ref{thm:AKS_sorting}}$ such that the following holds.
For every $n$, there exists a parallel sorting network $(R, \mathcal{C})$ with $n$ registers and depth at most $C_{\ref{thm:AKS_sorting}}\log n$. 
\end{theorem}

The earliest use of this result in extremal graph theory, that the authors are aware of, is by K\"{u}hn, Lapinskas, Osthus and Patel~\cite{kuhnlapinskas2014}. This result has also been recently used in a group theoretic setting in \cite{muyesser2022random}. The key challenge here, compared to \cite{kuhnlapinskas2014} and \cite{muyesser2022random}, is that the underlying host graphs we are working with are quite sparse.  

\par The next lemma gives a construction which provides a graph theoretic analogue of the comparison operation in a sorting network. The key feature is that the graph constructed has arbitrarily large girth, and can be found inside a sparse expander graph. In particular, the graph will be \textit{path-constructible}. 

\begin{definition}\label{Def_constructible}
    Let $G$ be a graph and let $A\subseteq V(G)$. We say that $G$ is $A$-path-constructible if there exists a sequence of edge-disjoint paths $P_1,\ldots, P_t$ in $G$ with the following properties.
    \begin{enumerate}[label= \upshape(\roman{enumi})]
        \item $E(G)=\bigcup_{j\in [t]} E(P_j)$.
        \item  For each $i\in [t]$, the internal vertices of $P_i$ are disjoint from $A\cup \bigcup_{j\in [i-1]}V(P_j)$.
        \item For each $i\in [t]$, at least one of the endpoints of $P_i$ belongs to $A\cup \bigcup_{j\in [i-1]}V(P_j)$.
    \end{enumerate}
\end{definition}

\begin{figure}
    \centering
    \includegraphics[width=0.7\textwidth]{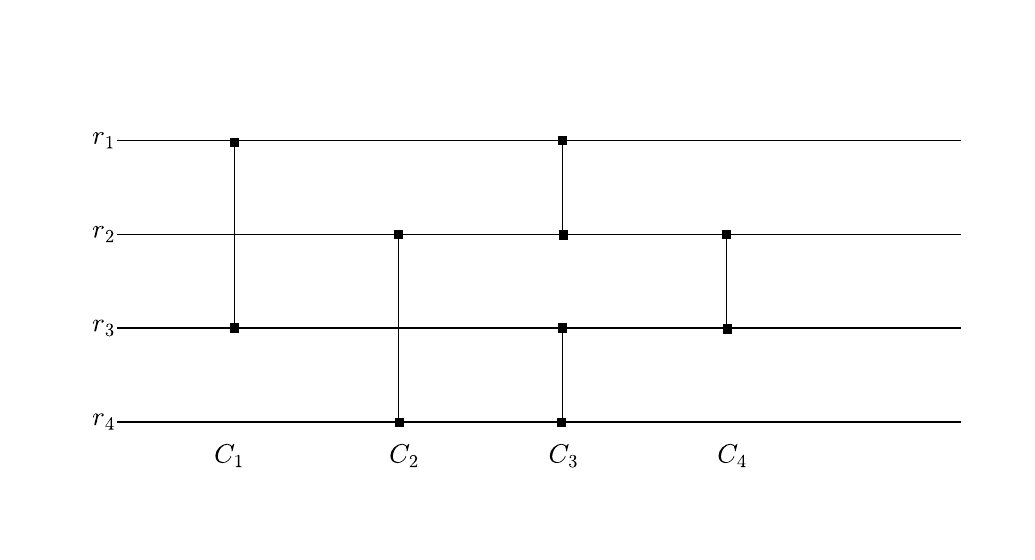}
    \caption{A depiction of a parallel comparison network with $4$ registers and depth $4$. 
    This is in fact a parallel sorting network. To see this, note that after the first $3$ levels, 
    the value $1$ will always be at $r_1$, and the value $4$ will always be at $r_4$, and the final level sorts the middle two layers.}
    \label{fig:sortingnetwork}
\end{figure}

\begin{lemma}\label{lemma:gadgets_exist_more_abstract}
  For every $k\in\mathbb N$ with $k \equiv 2\mod 4$, there exists a graph $G_k$, with distinct vertices $v^{\inn}_1,v^{\inn}_2,v^{\out}_1,v^{\out}_2$ and paths $P_1,P_2,Q_1,Q_2\subset G_k$, such that the following properties hold.
  \begin{enumerate}[label=\upshape(\roman{enumi})]
    \item \label{it:small} $|G_k|= 2k(k-1)$.
    \item\label{it:maxdeg} $\Delta(G_k) = 3$.
     \item\label{it:decompose} For each $z\in \{v^{\inn}_1, v^{\inn}_2\}$, $G_k$ is $\{z\}$-path-constructible with paths of length between $k$ and $2k+1$.
    \item \label{it:ends} $P_1$ is a $v^{\inn}_1, v^{\out}_2$-path, $Q_1$ is a $v^{\inn}_2, v^{\out}_1$-path, $P_2$ is a $v^{\inn}_1, v^{\out}_1$-path and $Q_2$ is a $v^{\inn}_2, v^{\out}_2$-path.
    \item \label{it:cover} For each $i \in [2]$, $V(G_k) = V(P_i) \ \dot{\cup} \ V(Q_i)$, i.e. the vertex set of $G_k$ is partitioned by $V(P_i)$ and $V(Q_i)$.
    \item \label{it:size} $|V(P_1)| = |V(P_2)| = |V(Q_1)| = |V(Q_2)| = k(k-1)$.
  \end{enumerate}
\end{lemma}
\begin{proof}We will construct $G_k$ from a cycle $C$ of length $2k$, by adding paths $P(ab)$ between specific pairs of vertices $a,b \in V(C)$. Let $U = \{u_1,\ldots, u_{k}\}$ and $V = \{v_1,\ldots,v_k\}$ be disjoint sets of vertices, and let $C$ be the cycle on $U \cup V$ obtained by adding the edges $u_1v_2, u_kv_{k-1}$; 
    $u_i u_{i+1}$ and $v_i v_{i+1}$, for every odd $i \in [k-1]$;
    $u_i v_{i+2}$, for every even $i \in [k-2]$;
    and $u_i v_{i-2}$ for every odd $i$ with $3 \leq i \leq k$. 
    For example, when $k=2$, we have the $4$-cycle $C=u_1u_2v_1v_2u_1$, and, when $k = 6$, we have $C = u_1u_2v_4v_3u_5u_6v_5v_6u_4u_3v_1v_2u_1$.

    Setting $S := \{(u_i,u_{i+1}): i \in [k-2], i \text{ even}\} \cup \{(v_i,v_{i+1}): i \in [k-2], i \text{ even}\}$,
    $G_k$ is obtained from $C$ by adding a path $P(ab)$ between $a$ and $b$, using $2k$ new vertices, for each pair $(a,b)\in S$. 
    For each $(a,b)\in S$, let $P(ba)$ be the path $P(ab)$ traversed in the opposite direction, that is, it is a path beginning with vertex $b$ and ending at $a$. Set $v^{\inn}_1:= u_1$, $v^{\inn}_2:= v_1$,  $v^{\out}_1 := u_k$ and $v^{\out}_2 := v_k$. Now define 
    \[\begin{array}{lll}
    P_1 &:=& u_1 P(v_2v_3)P(u_5u_4)\ldots P(v_{k-4}v_{k-3})P(u_{k-1}u_{k-2})v_k,\vspace{.2cm}\\ 
    Q_1&:=& v_1 P(u_3u_2)P(v_4v_5)\ldots P(u_{k-3}u_{k-4})P(v_{k-2}v_{k-1})u_k,\vspace{.2cm}\\
    P_2&:=& u_1 P(u_2u_3)P(u_4u_5)\ldots P(u_{k-2}u_{k-1})u_k, \vspace{.2cm}\\
    Q_2&:=& v_1 P(v_2v_3) P(v_4 v_5) \ldots P(v_{k-2} v_{k-1}) v_k.
    \end{array}\]

    We now check that $G_k$ has the desired properties. Observe that Properties~{\it\ref{it:small}}, {\it\ref{it:maxdeg}} and {\it\ref{it:ends}} follow by construction. 
    To verify {\it\ref{it:decompose}}, let $P_1'$ and $P_2'$ be edge-disjoint paths of length $k$, with the same endpoints, such that $P_1'\cup P_2'=C$, and $z$ is one of the endpoints of $P_1'$. Also, let paths $P_j'$, $j\ge 3$, correspond to the paths $P(ab)$ in an arbitrary order. By construction, the edge set of $\bigcup_{j\ge 1}P_i'$ is equal to $E(G)$, and each path $P_j'$ in the sequence meets the vertices of the previous paths $P'_{j'}$, $j'<j$, only at the endpoints of $P'_{j}$, if they intersect. Hence $G$ is \{z\}-path-constructible. For {\it\ref{it:cover}}, observe that, for each even $i$, $P(u_iu_{i+1})$ appears in $P_2$ and $P(u_{i+1} u_i)$ appears in $P_1 \cup Q_1$.
    Similarly, for each even $i$, we have that $P(v_i v_{i+1})$ appears in $Q_2$ and $P(v_i v_{i+1})$ appears in $P_1 \cup Q_1$.
    As all vertices in $G_k \setminus \{u_1,v_1,u_k,v_k\}$ appear in some $P(ab)$, we clearly have that {\it\ref{it:cover}} holds. 
    Finally, {\it\ref{it:size}} holds as every $P(ab)$ contains $2k+2$ vertices,
    $\frac{k}{2}-1$ of these paths appear in each of $P_1,Q_1,P_2,Q_2$ and there are $4$ `end' vertices $u_1,v_1,u_k,v_k$.\end{proof}

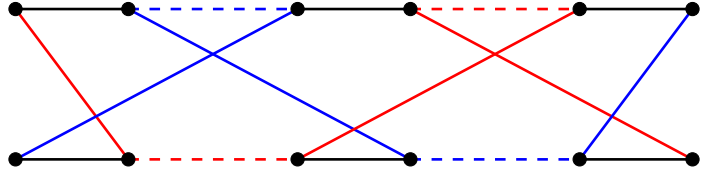
\begin{figure}[!ht]
\begin{center}
\begin{tikzpicture}
\draw [line width=1pt, color=red] (0,6)-- (1.5,4);
\draw [line width=1pt, color=blue] (1.5,6)-- (5.25,4);
\draw [line width=1pt, color=red] (5.25,6)-- (9,4);
\draw [line width=1pt, color=blue] (7.5,4)-- (9,6);
\draw [line width=1pt, color=blue] (3.75,6)-- (0,4);
\draw [line width=1pt, color=red] (7.5,6)-- (3.75,4);
\draw [line width=1pt] (0,6)-- (1.5,6);
\draw [line width=1pt] (0,4)-- (1.5,4);
\draw [line width=1pt] (3.75,4)-- (5.25,4);
\draw [line width=1pt] (3.75,6)-- (5.25,6);
\draw [line width=1pt] (7.5,6)-- (9,6);
\draw [line width=1pt] (7.5,4)-- (9,4);
\draw [line width=1pt,dash pattern=on 4pt off 4pt, color=blue] (1.5,6)-- (3.75,6);
\draw [line width=1pt,dash pattern=on 4pt off 4pt, color=red] (5.25,6)-- (7.5,6);
\draw [line width=1pt,dash pattern=on 4pt off 4pt, color=red] (1.5,4)-- (3.75,4);
\draw [line width=1pt,dash pattern=on 4pt off 4pt, color=blue] (5.25,4)-- (7.5,4);
\begin{scriptsize}
\draw [fill=black] (0,6) circle (2.5pt);
\draw [fill=black] (1.5,4) circle (2.5pt);
\draw [fill=black] (1.5,6) circle (2.5pt);
\draw [fill=black] (0,4) circle (2.5pt);
\draw [fill=black] (3.75,6) circle (2.5pt);
\draw [fill=black] (5.25,6) circle (2.5pt);
\draw [fill=black] (7.5,6) circle (2.5pt);
\draw [fill=black] (9,6) circle (2.5pt);
\draw [fill=black] (7.5,4) circle (2.5pt);
\draw [fill=black] (9,4) circle (2.5pt);
\draw [fill=black] (5.25,4) circle (2.5pt);
\draw [fill=black] (3.75,4) circle (2.5pt);
\end{scriptsize}
\end{tikzpicture}
\caption{The graph $G_6$, as given by Lemma~\ref{lemma:gadgets_exist_more_abstract}. The path in red is $P_1$ and the path in blue is $Q_1$. The dotted lines correspond to the added $P(ab)$ paths and the full lines form a copy of $C_{12}$. The top vertices represent $U$ and the bottom vertices represent $V$.}\label{fig:gadget}
\end{center}
\end{figure}
Using an optimal sorting network as a template, we will `glue together' copies of the graph $G_k$ from Lemma~\ref{lemma:gadgets_exist_more_abstract} to obtain the following result. 
\begin{proposition}\label{prop:sorting_network_template}
   There exists a constant $C_{\ref{prop:sorting_network_template}}$ such that for every $n$ and $k$ and every $\ell\geq \lfloor C_{\ref{prop:sorting_network_template}}k^2\log n\rfloor$, there exists a graph $G$, with $\Delta(G) = 4$, which contains disjoint subsets $A,B\subseteq V(G)$ with $|A|=|B|=n$, such that the following properties hold. 
\begin{enumerate}[label=\upshape(\roman{enumi})]
  \item $G$ is $A\cup B$-path-constructible with paths of length between $k$ and $4k$ (recall Definition~\ref{Def_constructible}).
  \item For any bijection $\phi\colon A \to B$, there exists a $P_\ell$-factor in $G$ such that each path has as endpoints some $a\in A$ and $\phi(a)\in B$.
\end{enumerate}
\end{proposition}
\begin{proof}In what follows, whenever we add a structure during our construction,  which will be either a copy of $G_k$ or a path, 
    the only intersection with the previously existing vertices and edges will be that explicitly mentioned. Also, we will construct $G$ for some $\ell=O(k^2\log n)$ as this implies the existence of a constant $C_{4.4}$ and some graph $G'$ with $\ell'$ exactly equal to any value above $\lfloor C_{\ref{prop:sorting_network_template}}k^2\log n\rfloor$ by adding additional paths of length between $k$ and $4k$ to extend the $P_{\ell}$-factor in $G$.
    \par By Theorem~\ref{thm:AKS_sorting}, there exists a parallel sorting network $(R,\mathcal{C})$ with $n$ registers and depth $\ell' \leq C_{\ref{thm:AKS_sorting}}\log n$. Since we are only constructing $G$ for some $\ell=O(k^2\log n)$, we may without loss of generality (and without relabelling) increase the value of $k$ by at most $3$, if necessary, so that $k\equiv 2$ (mod $4$). Let $G_k$ be the graph given by Lemma~\ref{lemma:gadgets_exist_more_abstract}.
    Firstly, we define  $A^{\out}_0:=A$ and $A^{\inn}_{\ell'+1}:=B$, and, for $i\in[\ell']$, we let $A^{\inn}_{i}$ and $A^{\out}_{i}$ be sets of size $n$, all disjoint from each other. For each $1\leq i\leq \ell'+1$, we label the vertices of the sets as $A^{\inn}_{i} =: \{v^{\inn}_{i,1}, \ldots, v^{\inn}_{i,n}\}$ and for each $0\leq i\leq \ell'$ we label $A^{\out}_{i} =: \{v^{\out}_{i,1}, \ldots, v^{\out}_{i,n}\}$. 
    
    We now define the graph $G$. For each $i\in [\ell']$ and each pair of registers $(r_{j_1},r_{j_2})\in C_i$, we add a graph $G_{i, r_{j_1}, r_{j_2}}$ that is a copy of $G_k$, where $v^{\inn}_{i,j_1}$ and $v^{\out}_{i,j_1}$   correspond to $v_1^{\inn}$ and $v_1^{\out}$, respectively, and $v^{\inn}_{i,j_2}$ and $v^{\out}_{i,j_2}$ correspond to $v_2^{\inn}$ and $v_2^{\out}$, respectively.
    Secondly, given $C_i\in \mathcal{C}$, for each register $r_{j}$ that is not included in any pair of registers from $C_i$, we add a $v_{i,j}^{\inn}, v_{i,j}^{\out}$-path $P_{i,j}$ with $k(k-1)$ vertices.  Finally, for each $0\le i\le\ell'$ and $j \in [n]$, we add a $v_{i,j}^{\out}, v_{i+1,j}^{\inn}$-path $Q_{i,j}$ of length $k$. 
    
    \par To check (i), we need to provide a sequence $\mathcal{S}$ of subpaths with the desired properties. We build $\mathcal{S}$ inductively. Order $C_1$ arbitrarily, and for each pair $(r_{j_1},r_{j_2})\in C_1$ (recall, by definition $j_1 < j_2$) add the path $Q_{0, j_1}$ to $\mathcal{S}$, noting that each of these paths has one endpoint in $A$. Afterwards, for each graph $G_{1, r_{j_1}, r_{j_2}}$, consider the sequence guaranteed by Lemma~\ref{lemma:gadgets_exist_more_abstract}\ref{it:decompose} with $z=v_1^\inn$, and append this sequence to $\mathcal{S}$. Now add to $\mathcal{S}$ each path $Q_{0, j'}$ not already added, followed by each path $P_{1,j}$ on $k(k-1)$ vertices, observing that each path when added intersects with $A_0^\out$ or $A_1^\out$ in at least one end vertex. We can proceed similarly for $i\geq 2$ to obtain the desired path sequence $\mathcal{S}$.

    \par Observe that, for each register $r_j$ and each level $i\in [\ell']$, we have two vertices, 
    an in-vertex and an out-vertex, contained in either a copy of $G_k$ or a path $P_{i,j}$, and thus, for all $i\in [\ell']$ and $j\in[n]$ there is a $v_{i,j}^{\inn}, v_{i,j}^{\out}$-path with $k(k-1)$ vertices. Therefore, taking into account the paths $Q_{i,j}$ as well, the number of vertices in $G$ is 
    \[n\cdot(\ell'k(k-1) +(\ell'+1) k)=O(nk^2\log n).\]
   It remains to confirm (ii) holds. Towards that goal, for a given bijection $\phi\colon A\to B$, we define a spanning subgraph $G_{\phi}\subseteq G$ that is a $P_\ell$-factor satisfying~(ii), where $\ell:= (|G|/n)-1=O(k^2\log n)$. We include in $G_\phi$ every edge from $G$ that belongs to one of the paths $Q_{i,j}$, for $0\le i\le \ell'$ and $j\in [n]$,
    and also include all edges belonging to any paths $P_{i',j'}$ that exist,
    for any $i'$ and $j'$. For each $G_{i, r_{j_1}, r_{j_2}}$, recall that $(r_{j_1}, r_{j_2})\in C_i$, 
    and so the values in registers $r_{j_1}$ and $r_{j_2}$ are either swapped or remain the same. 
    Reflecting this,
    if $C_i$ swaps the values in the registers we take the paths of the copy of $G_{i, r_{j_1}, r_{j_2}}$ corresponding to $P_1$ and $Q_1$ in the statement of Lemma~\ref{lemma:gadgets_exist_more_abstract},
    otherwise we take the paths corresponding to $P_2$ and $Q_2$.
\begin{figure}[h!]
    \centering
    \includegraphics[width=0.7\textwidth]{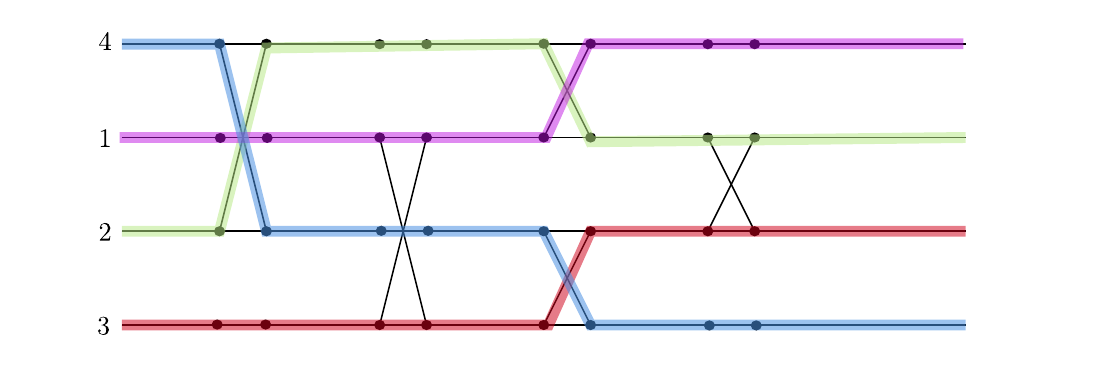}
    \caption{The graph $G$ obtained from the sorting network in Figure~\ref{fig:sortingnetwork} by replacing each pair of registers with a copy of $G_2$, the cycle of length $4$.  For the bijection $\phi$ mapping $1\to 4$, $2\to 1$, $3\to 2$, and $4\to 3$, the subgraph $G_\phi$ corresponds to the coloured paths.} 
    \label{fig:enter-label}
\end{figure}

It is easy to see that $G_\phi$ is a $P_{\ell}$-factor with all endpoints in $A$ and $B$.
    Furthermore, the pair of paths that correspond to swapping values, and the pairs of paths that correspond to maintaining the same values in the registers exactly mimic the operation of the parallel sorting network $(R, \mathcal{C})$ on $\phi$. 
    This guarantees that for each $a \in A$, there exists a path in the $P_{\ell}$-factor with endpoints $a\in A$ and $\phi(a) \in B$, as desired.
\end{proof}

\subsection{Proof of the sorting network lemma}
The remaining task is to show that the graph described in Proposition~\ref{prop:sorting_network_template} can be found inside a sparse expander. 
\begin{proof}[Proof of Lemma~\ref{lemma:find_sorting_network_in_expander}] Let $C_{\ref{lemma:find_sorting_network_in_expander}}=C_{\ref{prop:sorting_network_template}}$ and select $K$ so that $1/K\ll 1/C_{\ref{lemma:find_sorting_network_in_expander}}$.
Let $G'$ be the graph given by Proposition~\ref{prop:sorting_network_template} with $n_{\ref{prop:sorting_network_template}}=|V_1|$, $k=\log n$, and $\ell=\lfloor C_{\ref{lemma:find_sorting_network_in_expander}}\log^3 n \rfloor$ (note $\ell\geq C_{\ref{prop:sorting_network_template}}k^2\log |V_1|$). 
We claim that $G'$ can be embedded into $G$, with $A$ copied to $V_1$ and $B$ copied to $V_2$. Clearly, this would imply the lemma, as $G'$ satisfies property (ii) from Proposition~\ref{prop:sorting_network_template}.

We will embed $G'$ using property (i) from Proposition~\ref{prop:sorting_network_template} and repeatedly applying either of Lemmas~\ref{lemma:extendable:embedding} or \ref{lemma:connecting}, so that at each stage we embed a path $P$ of length $q\in [\log n, 4\log n]$ with all its internal vertices disjoint from all the vertices embedded in previous stages. 
Let $P_1,\ldots, P_t$ be the sequence of subpaths guaranteed by Proposition~\ref{prop:sorting_network_template}.

 Letting $P_0=I(A\cup B)$, at stage $i$ of the embedding, for $0\le i\le t$, we will have a $(D,m)$-extendable subgraph $S_i\subset G$ such that $S_i$ is a copy of $P_0\cup\ldots\cup P_i$, $V_1$ is copied to $A$, and $V_2$ is copied to $B$. Let $S_0=I(V_1\cup V_2)$, with $A$ copied to $V_1$ and $B$ copied to $V_2$, and suppose that we have successfully performed stage $i$ of the embedding, for some $0\le i<t$, and let us demonstrate how we continue the embedding in stage $i+1$. Recall that, by property of the sequence, $P_{i+1}$ is a path which is internally disjoint with $P_0\cup P_1\cup\ldots\cup P_i$ and at least one of the endpoints of the path belongs to $P_0\cup P_1\cup\ldots\cup P_i$. Let $x_{i+1},y_{i+1}$ be the endpoints of $P_{i+1}$. If both $x_{i+1}$ and $y_{i+1}$ belong to $\bigcup_{j\in [i]}V(P_j)$, we will use Lemma~\ref{lemma:connecting} to find a copy $Q_{i+1}$ of $P_{i+1}$ connecting the image of $x_{i+1}$ and $y_{i+1}$. If only one of the endpoints of $P_{i+1}$ is already embedded, say $x_{i+1}$, then we will use Lemma~\ref{lemma:extendable:embedding} to find a copy $Q_{i+1}$ of $P_{i+1}$ starting from the image of $x_{i+1}$. In either case, we extend the embedding by letting $S_{i+1}=S_i+Q_{i+1}$, which is a $(D,m)$-extendable subgraph of $G$ where $A$ is copied to $V_1$ and $B$ is copied to $V_2$.

 We now show that the conditions to apply Lemmas~\ref{lemma:extendable:embedding} and~\ref{lemma:connecting} are satisfied when performing step $i+1$ of the embedding. Indeed, first note that $m \leq n/100D$ and 
 \begin{equation}\label{eq:sorting_network_expander}
     |S_i|+|P_{i+1}|\le |G'|\le |V_1|\cdot 2C_{\ref{lemma:find_sorting_network_in_expander}}\log^3 n\le \frac{n}{10}, \end{equation}
where we used that $|V_1|\le n/K\log^3n$ and $1/K\ll 1/C_{\ref{lemma:find_sorting_network_in_expander}}$. Moreover, note that we are adding a path of length at least
\[\log n\ge \log (2m)\ge \frac{6\log (2m)}{\log(D-1)},\]
 which implies that the length condition in Lemma~\ref{lemma:connecting} is satisfied. Therefore, as $\Delta(S_i)\le 5\le D/2$ and by~\eqref{eq:sorting_network_expander}, we can indeed use either Lemma~\ref{lemma:extendable:embedding} or Lemma~\ref{lemma:connecting} to perform the embedding in stage $i+1$. 

Finally, let us observe that after $t$ stages we have found a $(D,m)$-extendable subgraph $S_{t}\subset G$ which is a copy of $P_1\cup\ldots\cup P_t=G'$, where $A$ is copied to $V_1$ and $B$ is copied to $V_2$.  Thus, letting $S_{res}=S_t$ finishes the proof.
\end{proof}

\section{Proof of Theorem~\ref{thm:mainthm}}\label{sec:mainthm}
The following result allows us to assume an upper bound on $d$ in the proof of Theorem~\ref{thm:mainthm}.
\begin{theorem}[Koml\'os--S\'ark\"ozy--Szemer\'edi~\cite{KSS95}]\label{thm:KSS}\ For every $\Delta\in\mathbb N$ and $\varepsilon>0$,
there exists $n_0\in\mathbb N$ such that the following holds for all $n\ge n_0$.
If $G$ is an $n$-vertex graph with $\delta(G)\ge (1+\varepsilon)\frac{n}{2}$, 
then $G$ contains a copy of every $n$-vertex tree with maximum degree bounded by $\Delta$.\end{theorem}
We now give the proof of our main theorem, which follows closely the sketch given in Section~\ref{proof_overview}. However, there will be an additional complication in the proof, where we will need to address the following technical issue. Once we embed the tree with the internal vertices of the bare paths removed, we have no control over how the graph expands into the image of the endpoints of the bare paths (which is what we need afterwards in order to find the perfect matchings). To overcome this problem, in Step 4 in the proof we will find a collection of short paths starting from the endpoints of the bare paths and ending at two random subsets that we set aside at the beginning of the proof, after which we can sequentially find perfect matchings as described in Section~\ref{proof_overview}.

\begin{proof}[Proof of Theorem~\ref{thm:mainthm}] We first choose constants from `right to left' as follows.
\[1/n\ll 1/C\ll\varepsilon\ll1/K\ll\mu, 1/C_{\ref{thm:manyleaves}}, 1/C_{\ref{lemma:find_sorting_network_in_expander}},  1/C_{\ref{prop:sorting_network_template}},1/K_{\ref{thm:manyleaves}}\ll\alpha,1/\Delta.\]
Let $d\in\mathbb N$ and $\lambda>0$ so that $d/\lambda \ge C\log^3n$. Let $G$ be an $(n,d,\lambda)$-graph and note that,
by Theorem~\ref{thm:KSS}, we may assume $d\le 0.6n$. 
Then, from Lemma~\ref{lemma:secondeig}, we have
\begin{equation}\label{eq:lowerbound:d}
    d\ge \tfrac{n-d}{n-1}\cdot \left(C\log^3 n \right)^2\ge C\log^6n.
\end{equation}

\textbf{Step 0: Find a large collection of long bare paths in $T$.} We may assume that $T$ contains less than $K_{\ref{thm:manyleaves}}\lambda n/d \leq \alpha n/\log^3n$ leaves, as, 
otherwise, we can use Theorem~\ref{thm:manyleaves} to find a copy of $T$ in $G$. 
Then, Lemma~\ref{lemma:fewleaves} implies $T$ contains at least 
\[\frac{n}{1+2K\log^3 n}-\frac{2\alpha n}{\log^3 n} + 2\ge \frac{n}{4K\log ^3 n} \]
vertex-disjoint bare paths of length $2K\log^3n$. 
Therefore, for $k:=\lfloor\frac{\mu n}{K\log^3 n}\rfloor$, we may take a collection $P_1,\ldots, 
P_k\subset T$ of bare paths in $T$, each of length $\ell=K\log^3n$. 
Let $T'$ be the forest obtained by removing the edges of $P_1\cup\ldots \cup P_k$ from $T$. For clarity, $T'$ includes both endpoints of each path $P_i$. 

\textbf{Step 1: Set aside some random sets.} Pick four disjoint random subsets $R_1,R_2,R_3,R_4\subset V(G)$, each of size $k$,
so that with high probability we have that 
\begin{enumerate}[label = \upshape\textbf{A1}]
\item\label{A:1} $\frac{\mu d}{2K\log^3 n} \leq d(v,R_i)\le \frac{2\mu d}{K\log^3n}$ for all $v\in V(G)$ and $i\in[4]$,
\end{enumerate}
and pick another random set $R_0$, disjoint from $R_1,R_2,R_3,R_4$ and of size $\mu n/2$, so that with high probability we have
\begin{enumerate}[label = \upshape\textbf{A2}]
\item\label{A:2} $\mu d\geq d(v,R_0)\geq \mu d/ 4 $ for every $v\in V(G)$.
\end{enumerate}

Note that \ref{A:1} and \ref{A:2} both hold with high probability by Lemma~\ref{lem:concentration} (and a union bound) and \eqref{eq:lowerbound:d}. 
Set $m:=\ell - 10\lfloor\log n\rfloor - C_{\ref{lemma:find_sorting_network_in_expander}}\lfloor\log^3 n\rfloor - 1$ and $R:=\bigcup_{0\le i\le 4}R_i$. 
Observe that $|R| \leq \mu n$ and $m \leq K\log^3 n$. 
Then, since $1/n_0 \ll \varepsilon\ll 1/K$, $0 < \mu \leq \frac{1}{10}$ and $n\ge n_0$, we may use Lemma~\ref{lem:concentration} to find pairwise disjoint subsets $V_1,\ldots, V_m\subset V(G)\setminus R$ such that
\begin{enumerate}[label =\textbf{B\arabic{enumi}}]
    \item\label{B:1} $|V_i| = \frac{\varepsilon n}{\log^3n}$ for all $i\in [m]$, 
    \item\label{B:2} $\frac{\varepsilon d}{2\log^3 n} \leq d(v,V_i)\le \frac{2\varepsilon d}{\log^3n}$ for all $v\in V(G)$ and $i\in[m]$, and
    \item\label{B:3} $d(v,V(G)\setminus (R\cup\bigcup_{i\in[m]}V_i))\ge \frac{2d}{3}$ for all $v\in V(G)$.
\end{enumerate}
\textbf{Step 2: Set aside a sorting network.} Set $t:=\lfloor C_{\ref{lemma:find_sorting_network_in_expander}}\log^3 k\rfloor$, noting $t \leq \ell/100$, and let $R_3=:\{x_1,\ldots, x_k\}$ and $R_4=:\{y_1,\ldots, y_k\}$.
\begin{claim}\label{claim:embed_sorting_network}There is a subset $W\subset V(G)\setminus \left(R\cup \bigcup_{i\in [m]} V_i\right)$ with $|W|=k(t - 1)$ 
such that the following holds. 
\begin{enumerate}[label = \upshape\textbf{C}]
    \item\label{C} For every bijection $\varphi:[k]\to [k]$ there is a collection of vertex-disjoint paths $D_1,\ldots, D_{k}$ such that, 
    for each $i\in [k]$, $D_i$ is an $x_i,y_{\varphi(i)}$-path of length $t$ whose interior vertices are in $W$ (and thus partition $W$).
\end{enumerate}
\end{claim}
\begin{proof} Set $\hat{G}:= G-(R_0\cup R_1\cup R_2 \cup \bigcup_{i\in [m]} V_i)$. Observe that $|\hat{G}|\geq n/2$, that $\lambda n/d \leq |\hat{G}|/(100\log |\hat{G}|)$ since $d/\lambda \geq C\log^3 n$ and $n$ is sufficiently large,
and that $2\lambda n/d\le k \leq |\hat{G}|/K \log^3 |\hat{G}|$ as $1/C\ll \mu, 1/K$. We aim to apply Lemma~\ref{lemma:find_sorting_network_in_expander} with the following parameters: $G_{\ref{lemma:find_sorting_network_in_expander}} = \hat{G}$ , $m_{\ref{lemma:find_sorting_network_in_expander}}=: m'= \lambda n / d$, $D=\log n$, $V_1=R_3$ and $V_2=R_4$. 
Note that $\hat{G}$ is $m'$-joined as $G$ is $m'$-joined, and also $D\geq 100$,
so it suffices to check that $I(R_3\cup R_4)$ is $(D,m')$-extendable in $\hat{G}$.
By \ref{A:1}, \ref{A:2} and \ref{B:3}, every vertex of $\hat{G}$ has at least $d/1000 \gg \mu d/\log^2 n$ neighbours in $\hat{G}- (R_3\cup R_4)$,
so by Lemma~\ref{lemma:expansion2}, with $z_{\ref{lemma:expansion2}} = \log n$, and Proposition~\ref{prop:weakerextendability}, we verify the definition of $(D,m')$-extendable for sets $U$ of size at most $m'$.
For a set $U$ with $m'\le |U|\le 2m'$, $(D,m')$-extendability follows by $m'$-joinedness. Indeed, we have that $|N_{\hat{G}}(U)\setminus (R_3\cup R_4)|\geq |\hat{G}|-1-|R_3\cup R_4| - 3m' \geq n/100 \geq 2Dm'$, as needed.

\par Now, the conclusion of Lemma~\ref{lemma:find_sorting_network_in_expander} corresponds exactly to \ref{C}.
\end{proof}
\textbf{Step 3: Embed $T'$, an almost spanning forest.} Let $G'=G-(W\cup R_3\cup R_4\cup\bigcup_{i\in[m]}V_i)$ and note that 
\begin{equation}\label{eq:G':size}
    |G'|\overset{\ref{B:1}}{\ge} n-|W\cup R_3\cup R_4|-m\cdot\frac{\varepsilon n}{\log^3 n}\ge n-k\cdot (t+1)-K\varepsilon n,
\end{equation}
where we have used that $\ell=K\log^3 n$ and $m\leq \ell$. 

For some $v\in V(G')\setminus (R_1\cup R_2)$ chosen arbitrarily, define $S_v := I(\{v\}\cup R_1\cup R_2)$.

\begin{claim}\label{claim:v_R_1_R_2_is_extendable} $S_v$ is a $(10\Delta,\frac{\lambda n}{d})$-extendable subgraph of $G'$. \end{claim} \begin{proof} 
Indeed, for a subset $X\subset V(G')$ with $1\le |X|\le \frac{\lambda n}{d}$, we have
\[|N_{G'}(X)\setminus V(S_v)|\ge |N_{G}(X)\cap R_0|{\ge}10\Delta|X|,\]
where we have used \ref{A:2} and {Lemma}~\ref{lemma:expansion2}, with $z_{\ref{lemma:expansion2}} = 1$ in the second inequality.
If $\frac{\lambda n}{d}\le |X|\le \frac{2\lambda n}{d}$, we use Lemma~\ref{lemma:joined} and \eqref{eq:G':size} to deduce that
\[|N_{G'}(X)\setminus V(S_v)|\ge |G'|-1-|R_1\cup R_2|-\frac{3\lambda n}{d}\overset{\eqref{eq:G':size}}{\ge} \frac{n}{2}-\frac{3\lambda n}{d}\ge 10\Delta |X|. \]
Also, clearly $0 = \Delta(S_v) \leq 10\Delta$.
Then, Proposition~\ref{prop:weakerextendability} implies that $S_v$ is $(10\Delta,\frac{\lambda n}{d})$-extendable in $G'$. \end{proof}

Recall that $T'$ is the forest obtained by removing the edges of the bare paths $P_1,\ldots,P_k$ from $T$. 
We may add dummy edges to $T'$ to think of it as a tree rather than a forest for the following application. Use Lemma~\ref{lemma:extendable:embedding} to find a copy $S$ of $T'$ in $G'-(R_1\cup R_2)$ (the root $t$ to be embedded on $v$ can be chosen arbitrarily) such that $S\cup S_v$ is $(10\Delta,\frac{\lambda n}{d})$-extendable in $G'$.
This can be done as $G'$ is $\frac{\lambda n}{d}$-joined (trivially, since $G$ is $\frac{\lambda n}{d}$-joined), $S_v$ is $(10\Delta,\frac{\lambda n}{d})$-extendable in $G'$ by Claim~\ref{claim:v_R_1_R_2_is_extendable}, $|S_v|\le \frac{\mu n}{3}$, and
\begin{equation}\label{eq:G':size:2}|G'| - (20\Delta + 3)\lambda n/d \overset{\eqref{eq:G':size}}{\ge} n-k\cdot (t+1)-K\varepsilon n - (20\Delta + 3)\lambda n/d \ge |T'|+\frac{\mu n}{2},\end{equation}
where we have used that $|T'| \leq n - \frac{\mu n}{2})$, that $k\cdot (t+1)\leq \frac{\mu n}{1000}$ since $C_{2.3} \ll K$, that $\frac{(20 \Delta + 3)\lambda}{d} \leq \frac{20 \Delta + 3}{C\log^3 n} \ll \mu$ and that $K\varepsilon\ll \mu$. 

\textbf{Step 4: Connect the endpoints of $\bigcup P_i$ to $R_1$ and $R_2$}. Let $R_1=\{a_1,\ldots, a_k\}$ and $R_2=\{b_1,\ldots, b_k\}$.
For each $i\in [k]$, let $u_i$ and $v_i$ be the endpoints of the path $P_i$, recalling that these vertices belong to $V(T')$, hence copies of these vertices are present in the embedding $S$ we produced earlier. We refer to the copies of these vertices as $u_i$ and $v_i$, $i \in [k]$, as well.
\par We now find a collection of vertex-disjoint paths $Q_1,\ldots, Q_k$ such that 
\begin{enumerate}[label =\textbf{D\arabic{enumi}}]
\item $\bigcup_{i\in[k]}Q_i$ is disjoint from $V(S)\cup R_2$,
\item $Q_i$ is a $u_i,a_i$-path of length $t':=5\lfloor\log n\rfloor$  for each $i\in [k]$, and
    \item\label{D} $S\cup I(R_1\cup R_2)\cup Q_1\cup\ldots \cup Q_{i-1}$ is $(10\Delta,\frac{\lambda n}{d})$-extendable  for each $i\in [k]$.
\end{enumerate}
Indeed, setting $ Q_0 := \emptyset$, we have that $S\cup I(\{v\}\cup R_1\cup R_2) \cup Q_0$ is $(10\Delta,\frac{\lambda n}{d})$-extendable in $G'$.
Then, we can find $Q_1,\ldots, Q_k$ using iteratively Lemma~\ref{lemma:connecting} while ensuring~\ref{D}.
This can be done as, by~\ref{D} for $i-1$, $S\cup I(\{v\}\cup R_1\cup R_2)\cup \bigcup_{j\in[i]}Q_{j-1}$ is $(10\Delta,\frac{\lambda n}{d})$-extendable in $G'$, 
has maximum degree at most $\Delta(T)\le\Delta$, we have $t' \geq 2\frac{\log(\lambda n/d)}{\log(10 \Delta + 1)} + 1$ and
\[|G'| - 10D\lambda n/d - (t' - 1) \overset{\eqref{eq:G':size}}{\ge} n-k\cdot (t+1)-K\varepsilon n - 100\Delta\lambda n/d - t' \ge |S|+|R_1\cup R_2|+|\cup_{i\in[k]}Q_i|,\] where the second inequality holds for essentially the same reasons the second inequality of \eqref{eq:G':size:2} holds.
Let $F=S\cup I(\{v\}\cup R_1\cup R_2)\cup \bigcup_{i\in [k]}Q_i$. Then $F$ is $(10\Delta,\frac{\lambda n}{d})$-extendable in $G'$ by \ref{D}.
Similarly as above, find a collection of vertex-disjoint paths $Q_1',\ldots, Q_k'$ such that 
\begin{enumerate}[label =\textbf{E\arabic{enumi}}]
\item $\bigcup_{i\in[k]}Q'_i$ is disjoint from $V(F)$, and
\item $Q'_i$ is a $v_i,b_i$-path of length $t'$  for each $i\in [k]$.
\end{enumerate}
\textbf{Step 5: Connect $R_1$ to $R_3$ and $R_2$ to $R_4$.} Note that up to this point, we have embedded $T'$ and the first and last $t'+1$ vertices of each path $P_1,\ldots, P_k$.
Let $S'$ be the current embedding. 
\begin{claim}There are two permutations $\psi$ and $\psi'$ of $[k]$ and a collection of vertex-disjoint paths $L_1,\ldots, L_k$, $L'_1,\ldots, L'_k$ in $G-W$ such that 
\begin{enumerate}[label =\upshape\textbf{F\arabic{enumi}}]
\item the interior vertices of $L_1,\ldots, L_k, L'_1,\ldots, L'_k$ are disjoint from $V(S')$,
    \item $L_i$ is an $a_i,x_{\psi(i)}$-path of length $m+1$ for each $i\in [k]$, and
    \item $L'_i$ is an $b_i,y_{\psi'(i)}$-path of length $1$ for each $i\in [k]$. 
\end{enumerate}
\end{claim}
\begin{proof}
To find $L_i'$, we simply need to find a perfect matching between $R_4$ and $R_2$, which is guaranteed by Lemma~\ref{lemma:matching}, as $|R_2|=|R_4|=k$ and \ref{A:1} implies that $\delta(G[R_2,R_4])\ge \varepsilon d/2\log^3n$.
Note that all but $mk$ vertices of $T$ are now embedded. Let $S''$ be $S'$ together with $L_1' \ldots, L_k'$, and $T''$ be the subtree of $T$ isomorphic with $S''$. Then, as $|V(G)\setminus (W\cup V(S''))|=mk$, we can partition $V(G)\setminus(W\cup V(T''))=V_1'\cup\ldots\cup V_m'$ so that $V_i\subseteq V_i'$ for each $i\in [m]$ and $|V_1'|=\ldots=|V'_m|=k$. Note that, by~\ref{B:2}, we have $\delta(G[V'_i,V'_{i+1}])\ge \varepsilon d/2\log^3n$ for each $1\le i<m$, and, because of~\ref{A:1}, we also have $\delta(G[R_1,V'_1]),\delta(G[R_3,V_m'])\ge \varepsilon d/2\log^3n$. Thus, invoking Lemma~\ref{lemma:matching} iteratively, we can find perfect matchings between $(R_1,V_1')$, $(V_1',V_2')$, $\ldots$, $(V_{m-1}',V_m')$, and $(V_m', R_3)$.
The unions of these matchings give the desired collection $L_1, \ldots, L_k$ of vertex-disjoint paths.
\end{proof}
\textbf{Step 6: Use the sorting network.} Finally, we can use~\ref{C} to embed the interior $t-1$ vertices of each of the paths $P_1,\ldots, P_k$, and thus complete the embedding of $T$. Indeed, each $u_i$ and $v_i$ are connected via a path to some element $a_{j_i}$ of $R_2$ and $b_{j_i'}$ of $R_4$, respectively, and it suffices to choose the bijection $\phi$ so that it maps $a_{j_i}$ to $b_{j_i'}$ for each $i\in [k]$.
\end{proof}
\section{Concluding remarks}\label{sec:concluding}
\par The results of Han and Yang \cite{han2022spanning} are formulated in the more general context of $(n,d)$-expanders. The statement of Lemma~\ref{lemma:find_sorting_network_in_expander} works also in this level of generality, and so our methods imply universality results for $(n,d)$-expanders as well, but we do not provide the formal details here.
\par In the proof of Theorem~\ref{thm:mainthm}, $d$-regularity is not used in any essential way. In particular, all degrees being in the range $(1\pm \gamma) d$ for some small $\gamma>0$ would also have been sufficient. Hence, we expect that our methods could show that $\mathbb{G}(n,p)$ is $\mathcal{T}(n,\Delta)$-universal whenever $p\geq C_\Delta\log^6 n/n$. However, this is not as strong as the previously mentioned result of Montgomery~\cite{montgomery2019spanning} that $p\geq C_\Delta \log n /n$ is sufficient (see also his earlier work \cite{montgomery2014embedding} showing that $p\geq C_\Delta \log^5 n/n$ is enough). 

\par As another illustration of the use of Lemma~\ref{lemma:find_sorting_network_in_expander}, we sketch how to find cycle factors in pseudorandom graphs (see \cite{han2021finding} and the references therein for more results in this direction). 
\begin{theorem}\label{thm:cycle-factor}There exist positive constants $K$ and $C$ such that the following holds for all sufficiently large $n\in\mathbb N$. Let $k\in\mathbb N$ satisfy $k=K\log^3 n$ and $k\mid n$. Then, any $(n,d,\lambda)$-graph $G$ with $\lambda\le d/C\log^3 n$ contains a $C_k$-factor.
\end{theorem}
\begin{proof}[Sketch of proof of Theorem~\ref{thm:cycle-factor}] Let $K$ be much larger than $C_{\ref{lemma:find_sorting_network_in_expander}}$ and let $C$ be large enough. Let $V_1$ and $V_2$ be two disjoint random subsets of $G$ of size $n/k$. Set $\ell=C_{\ref{lemma:find_sorting_network_in_expander}} \log^3 (n/k)$ and $t=k- \ell_{\ref{lemma:find_sorting_network_in_expander}}+1$. 
Take also disjoint random sets $V_i$, for $3\le i\le t$, each of size $\lfloor n/10k\rfloor $. Applying Lemma~\ref{lemma:find_sorting_network_in_expander}, with parameters $V_1, V_2$ and $G-\bigcup_{i\geq 3} V_i$, find a subgraph $S_{res}$,  disjoint from $\bigcup_{i\geq 3} V_i$, such that $S_{res}$ contains a $P_\ell$-factor connecting $V_1$ with $V_2$ in any given ordering. Let $V_{res}=V(S_{res})$ and, without relabelling, distribute all vertices of $V(G)\setminus (V_{res}\cup \bigcup_{i\geq 1} V_i)$ into  $\bigcup_{i\geq 3} V_i$ so that $V_i=n/k$ for each $i\in [t]$. Then, find perfect matchings (using Lemma~\ref{lemma:matching}) in $G[V_{t},V_{1}]$ and $G[V_i,V_{i+1}]$, for $2\le i<t$,
thus finding a $P_{t-1}$-factor $\mathcal{P}$ in $G\setminus V_{res}$ so that each path in $\mathcal P$ has both endpoints in $V_1$ and $V_2$, respectively. Labelling the vertices of $V_1$ and $V_2$ as $\{x_1,\ldots, x_{n/k}\}$ and  $\{y_1,\ldots, y_{n/k}\}$, respectively, so that each path in $\mathcal{P}$ has as endpoints $x_i$ and $y_i$ for some $i\in [n/k]$, we can simply invoke~Lemma~\ref{lemma:find_sorting_network_in_expander}, with $\phi$ defined as $\phi(x_i)=y_i$ for each $1\le i\le n/k$, to obtain the desired $C_k$-factor.
\end{proof}
\par Let us note that at the end of the above proof, we have a lot of flexibility in the way we choose the bijection $\phi:V_1\to V_2$, which guarantees a wider class of $2$-regular spanning subgraphs than claimed, including Hamilton cycles. Also, up to the exponent of the logarithm, this matches the best known condition on $\lambda$ that forces a Hamilton cycle \cite{glock2023hamilton, KrivelevichHamilton}~and, notably, this seems to be the first method that works in this regime which does not make use of the \textit{Pos\'a rotation-extension} technique. By starting the proof with finding some initial paths to ensure divisibility conditions, the same idea can be used to show that $G$ contains all $2$-factors with sufficiently large girth, but we do not provide details here.
\par Finally, let us remark that the condition $\lambda \leq d/C$ is sufficient to apply Lemma~\ref{lemma:find_sorting_network_in_expander} in the above proof. We use the stronger assumption that $\lambda\le d/C\log^3 n$ only to find a path-factor in the leftover graph with designated start and endpoints. There could be more efficient techniques to perform this latter step, meaning that Lemma~\ref{lemma:find_sorting_network_in_expander} could potentially be used to attack the conjecture of Krivelevich and Sudakov~\cite{krivelevich2006pseudo} that $\lambda \leq d/C$ is a sufficient condition for an $(n,d,\lambda)$-graph to be Hamiltonian.

\section*{Acknowledgements}
We would like to thank Richard Montgomery for organising a workshop at the University of Warwick titled \textit{Spanning subgraphs in graphs and related combinatorial problems} where this work began in July 2023.
\bibliographystyle{amsplain}
\bibliography{spanningtrees}

\end{document}